\documentclass[12pt]{amsart}
\usepackage[pagewise]{lineno}
\usepackage{amsmath, amscd, amsfonts, amssymb, graphicx, color}
\usepackage{amsthm}%
\usepackage{mathrsfs}%
\usepackage[bookmarksnumbered, colorlinks, plainpages]{hyperref}

\newtheorem{theorem}{Theorem}[section]
\newtheorem{lemma}[theorem]{Lemma}
\theoremstyle{definition}
\newtheorem{definition}[theorem]{Definition}

\newtheorem{remark}[theorem]{Remark}
\newtheorem{corollary}[theorem]{Corollary}

\newtheorem*{maintheorem}{Main Theorem}

\begin{document}
\title[The simple $\mathscr{B}_{\psi}$-groups]{The simple $\mathscr{B}_{\psi}$-groups}
\author[M. Baniasad Azad]{Morteza Baniasad Azad}

\address{Abolfazl Street, 22 Bahman Blvd., Bardsir, Kerman, Iran. Postal code:78416-64979. Phone number:+98-9214739579. \newline}
	
	\email{baniasad84@gmail.com}
\thanks{}



\begin{abstract}
In a finite group $ G $, $ \psi(G) $ denotes the sum of element orders of $ G $.	
A finite group $ G $ is said to be a $\mathscr{B}_{\psi}$-group if $ \psi(H) < |G| $ for any proper subgroup $ H $ of $ G $.

In \cite{Lazorec} Lazorec asked: "what can be said about the $\mathscr{B}_{\psi}$ property of the finite simple groups $ \operatorname{PSL}(2, q) $?" In this  paper, we  answer  this question for the case of not only the finite simple groups
$ \operatorname{PSL}(2, q) $
but also all other finite simple groups.
We show that if $ S $ is a finite simple group, such that $ S \neq Alt(n) $ for any $ n \geq 14 $, then $S$ is a $\mathscr{B}_{\psi}$-group.
\end{abstract}

\maketitle

\section{Introduction}\label{sec1}
Throughout this paper all groups are finite.
The cyclic group of order $ n $ is denoted by $ C_n $.
The order of $ g \in G $ is denoted by $ o(g) $ and the sum  of element orders of $ G $ is denoted by $ \psi(G)  $.
Some relations between the structure of the group $ G $ and $ \psi(G) $ are given in \cite{algebra,canada}.
A Group $G$ is called   $\mathscr{B}_{\psi}$-group if
$ \psi(H) < |G| $ for all
proper subgroups $H$ of $G$ \cite{Harrington}.

For more details about $\mathscr{B}_{\psi}$-group, we refer the reader to \cite{Harrington,Lazorec};
for example, the authors proved the following results:
\begin{theorem}\cite[Theorem 18]{Harrington}\label{Harrington}
	Let $ G $ be a finite abelian group. Then $ G $ is a 
	$\mathscr{B}_{\psi}$-group if and
	only if 
	$ G \cong C_{p^2} $ or $ G \cong C_p^n $, where $ p $ is a prime and $ n \geq 1 $.	
\end{theorem}
\begin{theorem}\cite[Theorem 2.5]{Lazorec}	 Let $ G $ be a finite nilpotent group. Then $ G $ is a $\mathscr{B}_{\psi}$-group if and	only if $ G \cong C_{p^2} $ or $ exp(G) = p $, where $ p $ is a prime.
\end{theorem}
In \cite{Lazorec}, 
Lazorec put forward the following question:
\\
\textbf{Question.} \cite[Question 3]{Lazorec} What can be said about the $\mathscr{B}_{\psi}$ property of the finite simple groups $ \operatorname{PSL}(2, q) $? 

In this paper, we give an answer to the above question for the case of   not only for $ \operatorname{PSL}(2, q) $ but also  for all simple groups:
\begin{maintheorem}\label{main}
	Let $ S $ be a finite  simple group, such that $ S \neq	Alt(n) $ for any $ n \geq 14 $. Then
	$S$ is a $\mathscr{B}_{\psi}$-group.
\end{maintheorem}

\textbf{Notation.}
For
a group $ G $ we denote by $ meo(G) $ the maximum order of an element of $ G $
and by $ m(G) $ the minimum
index  of a maximal subgroup of $ G $.
\[meo(G)=\max\{o(g)|g \in G\},\qquad m(G)=\min \{|G:M|| M \text{ is proper of } G\}.\]
Also, we denote by $ m_2(G) $ the second minimum
index  of a maximal subgroup of $ G $.
\section{\bf Main results}

\begin{definition}
	We say a group $G$ is a $meo$-group if $meo(G)\leq m(G)$.
\end{definition}
\begin{remark}
	For any non-trivial group $G$, we have $\psi(G)<|G|\cdot meo(G)$.
\end{remark}

\begin{lemma}\label{meo-group}
	Every non-trivial $meo$-group is a $\mathscr{B}_{\psi}$-group.
\end{lemma}
\begin{proof}
	Let $G$ be a non-trivial $meo$-group.
	Then for all
	proper subgroups $M$ of $G$,
	$meo(G)\leq |G:M|$.
	If $M=1$, then $ \psi(M) <G $. If $M\neq1$, then
	\[\psi(M) < |M| meo(M) \leq |M| meo(G)  \leq |M| |G:M|=|G|.\]
	Therefore $G$ is a $\mathscr{B}_{\psi}$-group.
\end{proof}

\begin{lemma}\label{mohem} \cite[Theorem 1.2]{Mohem}
	For a finite non-abelian simple group $S$, either $meo(\operatorname{Aut}(S))<m(S)/4$ or $ S $ is listed in Table 1.
	\begin{table}[h] \label{111}
		\caption{Exceptions in Lemma \ref{mohem}}
		\begin{tabular}{|ll|l|l|l|l|} 
			\hline$M_{11}$ & $M_{23}$ & $\operatorname{Alt}(n)$ & $\operatorname{PSL}(n, q)$ & $\operatorname{PSU}(3, 3)$ & $\operatorname{PSp}(6, 2)$ \\
			$M_{12}$ & $M_{24}$ & & & $\operatorname{PSU}(3, 5)$ & $\operatorname{PSp}(8,2)$ \\
			$M_{22}$ & $H S$ & & & $\operatorname{PSU}(4, 3)$ & $\operatorname{PSp}(4, 3)$ \\
			\hline
		\end{tabular}
	\end{table}
\end{lemma}
\begin{theorem}\label{PSL(n,q)}
	Let $n\geqslant2$ and $(n, q)\neq(2,2), (2,3)$ and  $ q $ is a power $ p^a $ of a prime $ p $. Then
	\begin{enumerate}
		\item 
		the simple groups
		$ \operatorname{PSL}(n,q) $, where $ (n, q) \neq (4, 2) $, are $meo$-groups.
		\item
		the simple groups $ \operatorname{PSL}(n,q) $ are $\mathscr{B}_{\psi}$-groups. 
	\end{enumerate}
\end{theorem}
\begin{proof}
	(1)	We show that $meo(\operatorname{PSL}(n, q))\leq m(\operatorname{PSL}(n, q))$ for $ (n, q) \neq (4, 2) $.
	Using \label{mininspslnq}\cite[Theorem 1]{Mazurovminindessimpleclassical},
	we have
	\[m(\operatorname{PSL}(n, q))=\left\{ \begin{matrix}
		\begin{matrix}
			8 & (n, q)=(4, 2)  \\
			6 & (n, q)=(2, 9)  \\
			q & n=2, q\in \{5, 7, 11\}  \\
			\frac{(q^n - 1)}{(q - 1)} & \text{o.w}  \\
		\end{matrix}  \\
	\end{matrix} \right.\]	
	Using \cite{gap}, we have $meo(\operatorname{PSL}(2,5))=5$,  $meo(\operatorname{PSL}(2,7))=7$, $meo(\operatorname{PSL}(2,11))=11$ and
	$meo(\operatorname{PSL}(2,9))=5$ and so  $\operatorname{PSL}(2,5), \operatorname{PSL}(2,7), \operatorname{PSL}(2,11) $ and $ \operatorname{PSL}(2,9)$ are $meo$-groups.
	Therefore we assume that $(n,q) \notin \{(2,5), (2,7), (2,11), (2,9), (4, 2)\}$.
	We know that $\operatorname{PSL}(n,q)$ is a subgroup of $\operatorname{PGL}(n,q)$. Therefore 
	$meo(\operatorname{PSL}(n,q))\leqslant meo(\operatorname{PGL}(n,q))$.
	By \cite[Corollary 2.7]{Mohem}, we see that
	$ meo(\operatorname{PGL}(n, q)) = (q^n - 1)/(q - 1) $. Thus 
	\[meo(\operatorname{PSL}(n,q))\leqslant (q^n - 1)/(q - 1) = m(\operatorname{PSL}(n, q)) \] and so we get the result.
	
	(2) First we assume that $(n, q)\neq (4, 2)$.  Using part (1) and Lemma \ref{meo-group}, we have the simple groups $\operatorname{PSL}(n, q)$  are $\mathscr{B}_{\psi}$-groups. 
	Now we assume that $(n, q)=(4, 2)$, we show that $\operatorname{PSL}(4,2)$ is a $\mathscr{B}_{\psi}$-group.  We know that $\operatorname{PSL}(4,2) \cong \operatorname{Alt}(8)$.
	Let $M$ be a maximal group in $\operatorname{PSL}(4,2) \cong \operatorname{Alt}(8)$.
	If $M \ncong A_7$, then
	by \cite[page 22]{atlas}, we have $|M|\leqslant 1344$  and therefore
	$$\psi(M) < |M| meo(M) \leqslant 1344 meo(\operatorname{PSL}(4,2))= 1344 \cdot 15 = 20160=|\operatorname{PSL}(4,2)|.$$
	If $M \cong \operatorname{Alt}(7)$, then by using GAP we have $\psi(\operatorname{Alt}(7))=12601< 20160=|\operatorname{PSL}(4,2)|$.
	Therefore $\operatorname{PSL}(4,2)$ is a $\mathscr{B}_{\psi}$-group. 
\end{proof}
\begin{proof}[{\textsc{\textbf{Proof  of the main theorem}}}]
	Let $S$ be a simple group. If $S$ is an abelian group, then $S\cong C_p$, where $p$ is prime. Therefore
	by Theorem \ref{Harrington}, $S$ is a $\mathscr{B}_{\psi}$-group. So we suppose that $S$ is a non-abelian simple group.
	If $S$ is a simple group rather than  the groups listed in Table 1, then
	by Lemma \ref{mohem}, we have $meo(\operatorname{Aut}(S))<m(S)/4$.
	Since $meo(S)\leq meo(\operatorname{Aut}(S))$ and $m(S)/4<m(S)$, we have $meo(S)<m(S)$, therefore $S$ is an $meo$-group and using Lemma \ref{meo-group},
	$S$ is a $\mathscr{B}_{\psi}$-group.
	Now, we consider the following cases (listed in Table 1):
	\begin{itemize}
		\item Let $S \in \{M_{11}, M_{12}, M_{22}, M_{23}, M_{24}, H S\}$. Then by using \cite{atlas} we have
		\begin{align*}	
			& meo(M_{11})=11=m(M_{11}), &meo(M_{12})=11<12=m(M_{12}),\\
			& meo(M_{22})=11<22=m(M_{22}), &meo(M_{23})=23=m(M_{23}),\\
			& meo(M_{24})=23<24=m(M_{24}), &meo(HS)=20<100=m(HS).
		\end{align*}		
		Therefore $S$ is an $meo$-group	ans so $S$ is a $\mathscr{B}_{\psi}$-group.
		\item Let $S = \operatorname{PSL}(n, q)$. Then by Theorem \ref{PSL(n,q)}, we get the result.
		\item Let $ S \in \{\operatorname{PSU}(3, 3), \operatorname{PSU}(3, 5), \operatorname{PSU}(4, 3)\} \cup \{ \operatorname{PSp}(6, 2), \operatorname{PSp}(8, 2), \operatorname{PSp}(4, 3)\}$.
		Using \cite[Table 4]{Mohem}, \cite{atlas,gap}, we have:
		\begin{align*}	
			& meo(\operatorname{PSU}(3, 3))=12<28=m(\operatorname{PSU}(3, 3)),\\ 
			& meo(\operatorname{PSU}(3, 5))=10<50=m(\operatorname{PSU}(3, 5)), \\
			& meo(\operatorname{PSU}(4, 3))=12<112=m(\operatorname{PSU}(4, 3)),\\	
			& meo(\operatorname{PSp}(6, 2))=15<28=m(\operatorname{PSp}(6, 2)),\\
			& meo( \operatorname{PSp}(8, 2))=30<120=m( \operatorname{PSp}(8, 2)), \\
			& meo(\operatorname{PSp}(4, 3))=12<27=m(\operatorname{PSp}(4, 3)).
		\end{align*}
		Therefore $S$ is an $meo$-group and then $S$ is a $\mathscr{B}_{\psi}$-group.
		\item Let $ S =\operatorname{Alt}(n)  $, where $5 \leq n \leq 13$. Using \cite{atlas,gap} we have the following table:
		\[\begin{matrix}
			S & \operatorname{Alt}(5) & \operatorname{Alt}(6) & \operatorname{Alt}(7) & \operatorname{Alt}(8) & \operatorname{Alt}(9) & \operatorname{Alt}(10) & \operatorname{Alt}(11)   \\
			\hline
			meo(S) & 5 & 5 & 7 & 15 & 15 & 21 & 21    \\
			m(S) & 5 & 6 & 7 & 8 & 9 & 10 & 11   \\
			{{m}_{2}}(S) & 6 & 10 & 15 & 15 & 36 & 45 & 55  \\
			\psi(S) & 211 & 1411 & 12601 & 137047 & 1516831 & 18111751 & 223179001  \\ 
			|S| & 60 & 360 & 2520 & 20160 & 181440 & 1814400 & 19958400   \\ 
		\end{matrix}\]
		\[\begin{matrix}
			S & \operatorname{Alt}(12) & \operatorname{Alt}(13) & \operatorname{Alt}(14) & \operatorname{Alt}(15) \\
			\hline
			meo(S) & 35 & 35 &   &     \\
			m(S) & 12 & 13 &   &     \\
			{{m}_{2}}(S) & 66  & 78 &   &     \\
			\psi(S) & 2973194071 & 46287964867  & 835826439631  &  15722804528341  \\ 
			|S| & 239500800 & 3113510400  & 43589145600  &  653837184000  \\ 
		\end{matrix}\]
		By the above table we see that if $n \in \{5, 6, 7\}$, then $meo(\operatorname{Alt}(n))\leq m(\operatorname{Alt}(n))$  and therefore these groups are
		$meo$-groups and so by Lemma \ref{meo-group} these groups are $\mathscr{B}_{\psi}$-groups. 
		If $n \in \{8, 9, 10, 11, 12, 13\}$, then $\operatorname{Alt}(n)$ is not a $meo$-group. Let $M$ be a maximal subgroup of $\operatorname{Alt}(n)$. Now we consider two following cases:\\
		$\bullet$ Let $M \ncong \operatorname{Alt}(n-1)$. Then by the above table and \cite{atlas} 
		we have
		\begin{align*}
			\psi(M) < |M| meo(M) & \leq |M| meo(\operatorname{Alt}(n)) \leq |M|m_2(\operatorname{Alt}(n))\\
			& \leq |M| |\operatorname{Alt}(n):M|=|\operatorname{Alt}(n)|.	
		\end{align*}
		$\bullet$ Let $M \cong \operatorname{Alt}(n-1)$. Then by the above table we see that $\psi(\operatorname{Alt}(n-1))<|\operatorname{Alt}(n)|$.
		Therefore $\operatorname{Alt}(n)$, where $5 \leq n \leq 13$, is $\mathscr{B}_{\psi}$-group.
	\end{itemize}
	Thus we get the result.
\end{proof}
\begin{corollary}
	If $S$ is a simple group such that $S\neq \operatorname{PSL}(4,2)$, $\operatorname{Alt}(n)$, where $n\geq 8$, then $S$ is a $meo$-group.
\end{corollary}
\begin{corollary}
	Let $ G $ be a finite abelian group. Then $ G $ is a $meo$-group if and
	only if $ G \cong C_p^n $, where $ p $ is a prime.
\end{corollary}
\begin{corollary}	 Let $ G $ be a finite nilpotent group. Then $ G $ is a $meo$-group if and	only if  $ exp(G) = p $, where $ p $ is a prime.
\end{corollary}
\begin{remark}
	We know that group $\operatorname{Alt}(n)$ has a subgroup $M \cong \operatorname{Alt}(n-1)$, where $n\geq 3$.
	On the other hand,
	$\psi(\operatorname{Alt}(13))=46287964867 \nless 43589145600=|\operatorname{Alt}(14)|$,
	$\psi(\operatorname{Alt}(14))=835826439631
	\nless 653837184000=|\operatorname{Alt}(15)|$. Therefore $\operatorname{Alt}(14)$ is not a $\mathscr{B}_{\psi}$-group.
\end{remark}


\end{document}